\documentclass[leqno,10pt]{article}
\usepackage{amsmath}
\usepackage{amsfonts}
\usepackage{amssymb}
\usepackage[dvips]{graphicx}
\usepackage{eso-pic}

\usepackage{subfigure}

\usepackage{color}

\newtheorem{theorem}{Theorem}[section]

\newtheorem{corollary}[theorem]{Corollary}

\newtheorem{definition}{Definition}[section]

\newtheorem{lemma}[theorem]{Lemma}

\newtheorem{proposition}[theorem]{Proposition}

\newenvironment{proof}[1][Proof]{\noindent\textbf{#1.} }{\ \rule{0.5em}{0.5em}}

\newcommand{\ind}{1\hspace*{-0.25em}\mathrm{I}\hspace{0.2em}}
\newcommand{\R}{\mathbb{R}}

\renewcommand{\d}{\,\mathrm{d}}
\newcommand{\ds}{\displaystyle}
\newcommand{\supp}{\mathop{\rm supp}\,}
\renewcommand{\P}{\mathbb{P}}
\newcommand{\E}{\mathbb{E}}
\newcommand{\e}{\mathrm{e}}

\newcommand{\dy}{\,\mathrm{d}y}

\renewcommand{\d}{\,\mathrm{d}}
\newcommand{\dist}{\mathrm{dist}}

\newcommand{\eps}{\varepsilon}
\renewcommand{\leq}{\leqslant}
\renewcommand{\geq}{\geqslant}

\definecolor{mauve}{rgb}{0.4,0,0.4}
\newcommand{\mauve}[1]{{\color{mauve} #1}}
\let\margintemp=\marginpar
\renewcommand{\marginpar}[1]{\margintemp{\tiny\textbullet\;\mauve{#1}}}

\begin{document}

\

\begin{center}
{\Large \textbf{Large Deviations estimates for some non-local equations}\\ \vspace*{5mm}
\textbf{I. Fast decaying kernels and explicit bounds}}\\ \vspace*{0.5in}
{\large
\sc C. Br\"andle\footnote{\noindent  Departamento de Matem{\'a}ticas, U.~Carlos III de Madrid, 
28911 Legan{\'e}s, Spain\\e-mail: {\tt cristina.brandle@uc3m.es}} \& E. Chasseigne\footnote{Laboratoire de Math\'ematiques et Physique Théorique, U.~F. Rabelais, Parc de Grandmont, 37200 Tours, France\\
email: {\tt emmanuel.chasseigne@lmpt.univ-tours.fr}}}
\end{center}
\vspace*{0.5in}

{\small \begin{center}
\textsc{Abstract}
\end{center}

We study large deviations for some non-local parabolic type equations.
We show that, under some assumptions on the non-local term, problems
defined in a bounded domain converge with an exponential rate to the
solution of the problem defined in the whole space. We compute this
rate in different examples, with different kernels defining the non-local
term, and it turns out that the estimate of convergence depends strongly on
the decay at infinity of that kernel.

\

Keywords: Non-local diffusion, Large deviations, Hamilton-Jacobi equation, L\'evy operators.

\

Subject Classification: 47G20, 60F10, 35A35, 49L25
}


\section{Introduction}

Consider continuous and  bounded solutions
$u:\R^N\times[0,\infty)\to\R$ of the linear non-local equation
\begin{equation} \label{eq:0}
	\frac{\partial u}{\partial t}(x,t)=\int_{\R^N}J(x-y)u(y,t)\dy - u(x,t)\,\quad
	\text{in}\quad \R^N\times(0,\infty)\,,
\end{equation}
where $u(x,0)=u_0(x)$ is fixed, bounded and continuous. For simplicity, we will assume
throughout the paper that solutions are non-negative, and thus also $u_0\geq0$.
The kernel $J$ is assumed to be a symmetric, continuous probability density.

The main contribution
of this paper is to describe how solutions $u_R$ of~\eqref{eq:0}, but defined in the ball $B_R=\{x\in\R^N : |x|\leq R\}$,
converge to the solution $u$ of~\eqref{eq:0}.

Let us first explain what $u_R$ is exactly: we consider here the notion of
Dirichlet problem that consists in putting $u_R=0$ not only on the topological
boundary of $B_R$, but also in all the complement of $B_R\times[0,\infty)$. In this way,
$u_R$ solves the equation
\begin{equation}\label{eq:uR}
		\frac{\partial u_R}{\partial t}(x,t)=\int_{B_R}J(x-y)u_R(y,t)\dy - u_R(x,t) \quad \text{in}\
		B_R\times(0,\infty)
\end{equation}
with initial data $u_R(x,0)=u_0(x)$ in $B_R$.
\ We refer to \cite{ChasseigneChavezRossi07} and \cite{Chasseigne07}
for more information on these non-local Dirichlet problems, and also
\cite{BarlesChasseigneImbert07} for similar questions (with singular kernels).

As one can imagine, under suitable assumptions, $u_R$ will reasonably converge to $u$ as $R\to\infty$,
but we want to obtain some estimates of how fast convergence occurs. Actually, this problem may be seen as a numerical
question since of course, computing numerically solutions requires a bounded domain. In this case one has
to know how far from the real solution the computed one is.

In the case of the Heat Equation, the answer is given in \cite{BarlesDaherRomano94}: the distance between $u$ and $u_R$
in the ball of radius $\theta R$ (with $0<\theta<1$) is estimated by
\begin{equation*}
	\sup_{|x|\leq\theta R}(u-u_R)\leq \exp\Big(-R^2\frac{(1-\theta)^2}{4t} + o(1)\Big)\,,
\end{equation*}
which means that convergence occurs exponentially fast, with a rate of the order of
$R^2$ inside the exponential.

Our aim is to produce similar estimates
for non-local equations~\eqref{eq:0}.
We face here several difficulties, which imply non
trivial adaptations of ideas and techniques in~\cite{BarlesDaherRomano94}, that we list below:

\textit{i)} Various behaviours of $J$ imply various rates: the importance of the tail of $J$ enters into play, since the operator puts emphasis on the difference between $u$ and $u_R=0$ far from the point where we compute $u_R$,
as we shall explain below in the subsection devoted to the probabilistic aspects. Roughly speaking,
the more $J$ is big at infinity, the slower $u_R$ converges to $u$.

\textit{ii)} The structure of the Hamiltonian which describes the rate function $\mathcal{I}$ is completely
different: for the Heat Equation the associated Hamilton-Jacobi equation is
$ \partial_t \mathcal{I} + |\nabla \mathcal{I}|^2=0$,
hence $H(p)=p^2$. Here the problem for the rate function is related to the hamiltonian
\begin{equation}\label{eq:ham}
	H(p)=\int_{\R^N}\Big(\e^{p\cdot y}-1\Big)J(y)\dy\,,	
\end{equation}
which, although it is local, is the limit as $R\to\infty$ of hamiltonians involving a non-local term,
$$
H(x,v,p,M,\mathcal{L}_R[v])=-\int_{\R^N}\Big(\e^{-R\big\{v(x+y/R)-v(x)\big\}}-1\Big)J(y)\dy=-\mathcal{L}_R[v].
$$
This localization process is one of the main interesting features of this problem.

\textit{iii)} We are not facing here a diffusive effect, but more a transport effect: in the case of the Heat Equation, the scaling used in~\cite{BarlesDaherRomano94} in order to proof convergence is the parabolic $(Rx,t)$ one (equivalent to $(R^2x,Rt)$). Here, we have to use a hyperbolic change of variables $(Rx,Rt)$ in order to scale the problem in a suitable
way.

\

\textsc{Probabilistic context - }
The term ``large deviation'' comes from the french ``grands \'ecarts'' which was used first
to describe how far from the normal distribution, some exceptional events are. For instance,
it is well-known that if $(X_n)$ is a sequence of independent and identically distributed
random variables with $\mathbb{E}(X_i)=\mu$, then
\begin{equation*}
	\frac{X_1+X_2+\cdots+X_n}{n}\stackrel{\P}{\longrightarrow}\mu
\end{equation*}
as $n\to+\infty$, the convergence occurring in law (this is the law of large numbers). Now,
one may wonder how to estimate, for $\eps>0$ small, the quantity:
\begin{equation*}
	\mathbb{P}\Big(\Big|\frac{X_1+X_2+\cdots+X_n}{n}-
	\mu\Big|>\eps\Big).
\end{equation*}
A result of Cramer (1938, see~\cite{Hollander} for a proof), shows that if one defines the rate function
$$I(\eps):=\ds\lim_{n\to+\infty}-\frac{1}{n}\log
	\P\Big(\Big|\frac{X_1+X_2+\cdots+X_n}{n}-\mu\Big|>\eps\Big), $$
then
 $$\ds I(\eps)=\sup_{t\in\R}\big\{\eps t-\log M(t)\big\}\,,\ \text{where}\
  M(t)=\E[\e^{t X_1}]<\infty\,,$$
which implies,
$$\mathbb{P}\Big(\Big|\frac{X_1+X_2+\cdots+X_n}{n}-
	\mu\Big|>\eps\Big)\leq\e^{\ds -n\,I(\eps)}\,.$$
This exponential behaviour is typical of what is called ``large deviations''.

In this paper, $(u-u_R)(\cdot,T)$ mesures in some sense the total amount of process that can
escape from the ball $B_R$ between times $0$ and $T$ (see \cite{BarlesDaherRomano94} and \cite{ChasseigneChavezRossi07} for
more explanations about this aspect). Thus, our results may
be viewed as ``large deviations'' results in the sense that the probability of escaping
the ball $B_R$ up to a given time becomes small as $R\to\infty$. Exponentially small in fact,
with a rate which depends on the tail of $J$ since this tail measures the amount of
``big jumps''. Values of $J$ near the origin only concern ``small jumps'' that are not relevant
as far as escaping the ball is at stake. So this is why, as we explain in Section 5, adding
a singularity at the origin does not change the rate of convergence.

\

\textsc{Main results - } After a preliminary section in which some properties of non-local and Hamilton-Jacobi equations are reviewed, Section~\ref{sect:preliminaries}, we devote Section~\ref{sect:theoretical} to the theoretical behaviour of the problem.
Theorem \ref{thm:est-IR} is a general result that gives the following estimate
\begin{equation*}
	\lim_{R\to\infty}|u-u_R|(x,t)\leq e^{-R I_\infty(x/R,t/R)+o(1)R}
\end{equation*}
where the rate function $I_\infty$ satisfies the limit Hamilton-Jacobi problem associated
to the Hamiltonian $H$ defined in \eqref{eq:ham}:
\begin{equation}
    \label{eq:Iinfty}
		\begin{cases}
			\partial_t I_\infty + H(\nabla I_\infty)=0 & \text{in }B_1\times(0,\infty),\\
			I_\infty=0&\text{on }\partial B_1\times(0,\infty),\\
			I_\infty(x,0)=+\infty & \text{in }B_1.
		\end{cases}
\end{equation}
Using a Lax-Oleinik formula, see~\cite{LionsBook}, we obtain a semi-explicit expression for $I_\infty$, see~\eqref{est:IR}.

In Section~\ref{sect:examples} we study different cases where explicit computations can be derived. A typical and
interesting example, often considered by authors, concerns the case when $J$ is compactly
supported in, say, $B_\eta$. We prove that the behaviour of $I_\infty$ is of ``$s\ln s$''
type, which in turn implies that the rate function is of $R\ln R$ type;~i.e. for any $\theta\in(0,1)$,
\begin{equation*}
    \sup_{|x|\leqslant\theta R}|u-u_R|(t)\leqslant\exp\Big(-(1-\theta)\frac{R\ln ((1-\theta)R/t)}{\eta}
    + o(1)R\Big) \quad\text{as}\quad R\to\infty.
\end{equation*}

Other examples which imply different rates, like $R(\ln R)^{(\alpha-1)/\alpha}$ for
$J(y)\sim~\e^{-|y|^\alpha}$, are dealt with and the limit case $J(y)=\e^{-|y|}$ is also
considered even if this leads to a singular hamiltonian, defined only for $|p|<1$, a case not covered by the results of Section~\ref{sect:theoretical}.

Finally, in Section 5 we explain how to extend these results to more singular situations,
when the kernel is not integrable near the origin. Provided the decay at infinity remains reasonable
(i.e., the Hamiltonian is defined everywhere),
the presence of a singularity at the origin does not change the behaviour since as we have already mentioned,
only the tail of $J$ is important in estimating $(u-u_R)$.

Notice however that using kernels with singularities requires a suitable concept of solution;
we use here the notion of viscosity solutions derived in \cite{BarlesChasseigneImbert07} for L\'evy-type operators,
which allows us to handle this situation. The Hamiltonian also needs
to be modified a bit, as a corrector term appears in the equation:
\begin{equation*}
	H(p)=\int_{\R^N}\Big(\e^{p\cdot y}-1-\ind_{\{|y|<1\}}(p\cdot y)\Big)J(y)\dy\,.
\end{equation*}

\section{Preliminaries}
\label{sect:preliminaries}
\setcounter{equation}{0}

As a first step in order to prove the main theorem of this paper, Theorem~\ref{thm:est-IR}, we need to state some properties of the non-local equations we are considering here,~\eqref{eq:0} and~\eqref{eq:uR}. We also need to do a thorough study, see Subsection~\ref{subsect:hamiltonian.and.lagrangian}, of the hamiltonian $H$ and  of the related Hamilton-Jacobi equation~\eqref{eq:Iinfty}.
To this aim, let us first state the exact assumptions on kernel $J$:
\begin{equation}\label{cond:J}
	J\ \text{\ non-negative,\ \  radially symmetric, }\ \int_{\R^N}J(y)\dy=1.
\end{equation}
The case $\int J(y)\dy=c<\infty$ can be also be considered just by doing a change of variables in $t$.

These are natural conditions used to give sense to the non-local equation. But moreover,
we assume that $J$ has fast decay
at infinity:
\begin{equation}\label{eq:defH}
	\text{For any }p\in\R^N,\quad H(p)=\int_{\R^N}\Big(\e^{p\cdot y}-1\Big)J(y)\dy<+\infty\,.
\end{equation}
For instance, compactly supported kernels are authorized and non-compactly supported kernels also,
provided they decay sufficiently fast at infinity. The limit case is $J(y)=\e^{-|y|}$ for which
$H$ is finite only inside the ball $B_1$.
We also give a formal derivation of estimates
for this limit case, that we intend
to prove in a forthcoming paper.
Notice also that Section \ref{sect:levy} extends the results to kernels $J$ that are singular
at the origin provided they remain L\'evy measures, and \eqref{eq:defH} is satisfied (but integrating
on the domain $\{|y|>1\}$).

\subsection{Properties of the non-local equation}

Equations \eqref{eq:0} and \eqref{eq:uR} have been studied in \cite{ChasseigneChavezRossi07,Chasseigne07};
we refer to these papers for proofs of existence, uniqueness and comparison results, as well as other
qualitative properties like positivity up to the boundary.

However, we want  to deal here with only bounded (and continuous) initial data, a case not covered in
\cite{ChasseigneChavezRossi07} since the argument of Fourier transforms that was used there requires that both
$u_0$ and its Fourier transform are integrable. So let us first give a more general existence and
uniqueness result in $\R^N$.
\begin{definition}
    Let $u_0:\R^N\to\R$ be continuous, non-negative and bounded.
    A strong solution of \eqref{eq:0} with initial data $u(0,x)=u_0(x)$
    is a function $u\in\mathrm{C}^0\big(\R^N\times[0,\infty)\big)$
    such that $u_t\in\mathrm{C}^0\big(\R^N\times(0,\infty)\big)$ and~\eqref{eq:0}
    holds in the classical sense, pointwise.
\end{definition}

Before proving existence, we state a comparison result valid
in the class of bounded strong solutions:

\begin{proposition}
    \label{prop:comparison}
	Let $u,v$ be bounded strong solutions of \eqref{eq:0} with initial data $u_0$ and $v_0$
	respectively. If $u_0\leq v_0$, then $u\leq v$ everywhere.
\end{proposition}

\begin{proof}
	It is contained in \cite[thm 3]{BarlesImbert07}, using as L\'evy measure $\mu(z)=J(z)\,\mathrm{d}z$.
	The only adaptation is that we are here in a parabolic situation whereas the cited Theorem works
	for the elliptic case. The adaptation is standard, the ``$u_t$'' term replacing the
	``$-\gamma u$'' term in the Hamiltonian.
\end{proof}

\begin{theorem}\label{thm:ex-un}
	Let $u_0:\R^N\to\R$ be continuous non-negative and bounded. Then there exists a unique strong solution
	of~\eqref{eq:0} with initial data $u_0$.
\end{theorem}

\begin{proof}
	Consider a sequence $u_{0n}\in L^\infty(\R^N)\cap L^1(\R^N)$, so that the Fourier transform of $u_{0n}$
	is also integrable, and such that $u_{0n}\to u_0$ monotonically: one can choose a sequence
	of the form $u_{0n}=u_0\cdot \chi_n$, $\chi_n$ smooth, compactly supported and converging
	monotonically to $1$. We know from \cite{ChasseigneChavezRossi07} that there exists a unique solution
	$u_n$ of \eqref{eq:0} which is in fact continuous since its Fourier transform remains integrable.

	As $n$ increases, the sequence increases by comparison in the class of continuous and bounded
	solutions. Since $(u_n)$ is bounded, there exists a limit $u$ defined in all $\R^N\times(0,\infty)$.
	The same happens with the sequence $(v_n)$ defined as $v_n (x,t):= \e^t u_n(x,t)$, which satisfies
	the equation $\partial_t v_n = J\ast v_n$.

	Passage to the limit in the sense of distributions is done using dominated converge for the convolution
	term: $J\ast v_n$ converges to $J\ast v$ and the limit equation $\partial_t v=J\ast v$ is satisfied in
	the weak sense.
	
	Now, using well-known properties of the convolution,
    we have that $J\ast v$ is continuous (recall that
	$J$ is integrable while $v$ is bounded), so that $\partial_t v$ is a continuous function. The same holds for $u$
	and thus we recover a strong solution.

	Uniqueness follows from the above comparison principle, Proposition~\ref{prop:comparison}
\end{proof}

Similar arguments show that indeed $u_R$ converges to $u$:

\begin{proposition}Let $u_0:\R^N\to\R$ be continuous non-negative and bounded. For any $R>0$ let $u_R$ be the
solution of~\eqref{eq:uR} with initial datum $u_0$. Then the sequence $(u_R)$ converges monotonically as $R\to\infty$ to the solution
$u$ of \eqref{eq:0} with initial datum~$u_0$.
\end{proposition}

\begin{proof}
 We use the same proof as that of Theorem \ref{thm:ex-un}. The sequence $(u_R)$ is increasing with respect to $R$ and it is bounded by $\Vert u_0\Vert$. Hence it converges to some $u$ which is a distributional solution
 of \eqref{eq:0}. Using again properties of the convolution, we get that $u$ is in fact the unique
 strong solution of the equation with $u(x,0)=u_0(x)$.
\end{proof}

\subsection{Hamiltonians and Lagrangians}
\label{subsect:hamiltonian.and.lagrangian}

As it is well-known in the theory of Hamilton-Jacobi equations, the Legendre-Fenchel transform $L$ of $H$ defined as
\begin{equation}\label{def:lagrangian}
	L(q)=\sup\limits_{p\in\R^N}\big\{p\cdot q-H(p)\big\}\,,\quad q\in\R^N\,.
\end{equation}
plays an important role in representing solutions.
Since the rate function $I_\infty$ is the solution of a Hamiltin-Jacobi equation, see~\eqref{eq:Iinfty}  we review first some basic properties of both $H$ and $L$.
\begin{lemma}
    \label{Lemma:properties}
	Let $J$ satisfy \eqref{cond:J} and \eqref{eq:defH}. Then both $H$ and $L$ are
	nonnegative, increasing, strictly convex,
 superlinear and radially symetric.
\end{lemma}

\begin{proof}
	In order to prove rotation-invariancy, let $\sigma$ be a rotation of the sphere $\mathbb{S}^{N-1}$, then
	\begin{eqnarray*}
		H(\sigma p)&=&\int\big(\e^{(\sigma p)\cdot y}-1\big)J(y)\dy\\
		&=&\int\big(\e^{p\cdot(\sigma^{-1}y)}-1\big)J(y)\dy\\
		&=&\int\big(\e^{p\cdot y'}-1\big)J(\sigma y')\dy'\\
		&=&\int\big(\e^{p\cdot y'}-1\big)J(y')\dy'\\
		&=&H(p)
	\end{eqnarray*}
    (remember that $J$ is symmetric).
	
    Strict convexity and superlinearity
	come from the estimates:
	\begin{equation*}
		D^2 H(p)=\int \e^{p\cdot y}|y|^2 J(y)\dy>0\,,\ D_p H(0)=\int yJ(y)=0\,.
	\end{equation*}
	Indeed, by rotation-invariancy, $D_pH(p)$ is always pointing
	in the direction of $p$

 and moreover
    \begin{equation}
    \label{eq:DpH}
        p\cdot D_p H(p)=|p||D_pH(p)|>0
    \end{equation} for any $p\neq 0$. Using
	the strict convexity of $H$, we get that $H$ grows faster than linearly
	along the lines ${\lambda p}$, $\lambda>0$.

Strict convexity, together with $H(0)=0$ imply that $H$ is nonnegative.
	
    Finally, it is well-known,~\cite{Rockafellar},    that if $H$ enjoys the above properties, so does $L$.
\end{proof}

In the case we are considering,~i.e. $J$ radially symmetric, since $H$ and $L$ are also radially symmetric,
throughout the paper we denote by $H^*$ and $L^*$ the functions defined on $\R_+$ such that
\begin{equation}\label{def:sym-HL}
	H^*(|p|)=H(p)\quad\text{and}\quad L^*(|q|)=L(q).
\end{equation}

The following technical trick will be used several times in the paper,
hence we state it as a Lemma:
\begin{lemma}
	Let $J$ satisfy \eqref{cond:J} and \eqref{eq:defH}. For $q$ fixed, let $p_0=p_0(q)$ be a maximum point defining $L(q)$, that is,
	$L(q)=p_0\cdot q -H(p_0)$. Then $p_0\cdot q=|p_0||q|$.
\end{lemma}
\begin{proof}
	Notice first that by definition of $L$, $p_0$ solves the equation $q=D_p H(p_0)$.
	The result follows using~\eqref{eq:DpH}.
\end{proof}

\subsection{Viscosity solutions for the rate function equation}

In Section \ref{sect:theoretical}, we will have to study the following Hamilton-Jacobi equation with
Cauchy-Dirichlet boundary values,
\begin{equation}\label{pb:limitA}
		\begin{cases}
			\partial_t I^A + H(\nabla I^A)=0 & \text{in }B_1\times(0,\infty),\\
			I^A=0&\text{on }\partial B_1\times(0,\infty),\\
			I^A(x,0)=A & \text{in }B_1.
		\end{cases}
\end{equation}

Let us first recall the definition of  viscosity solutions for this equation
(see for instance \cite{CrandallLions83}):
\begin{definition}
\label{def:viscosity}
	A locally bounded u.s.c function $u:\overline{B}_1\times[0,\infty)\to\mathbb{R}$
 is a viscosity subsolution of
	\eqref{pb:limitA} if for any $\mathrm{C}^2$-smooth function $\varphi$, and any point
	$(x_0,t_0)\in\overline{B}_1\times[0,\infty)$ where $u-\varphi$ reaches a maximum, there holds,
	\begin{eqnarray*}
		x_0\in B_1 &\Rightarrow& \partial_t\varphi+H(\nabla\varphi)\leq0\text{ at }(x_0,t_0),\\
		x_0\in\partial B_1 &\Rightarrow & \min\big\{\partial_t\varphi+H(\nabla\varphi)\,;\,u\big\}
		\leq0\text{ at }(x_0,t_0)\,,\\
		t_0=0 &\Rightarrow & u\leq A\text{ in }B_1\,.
	\end{eqnarray*}
\end{definition}

A locally bounded l.s.c. function is a viscosity supersolution if the same holds with
reversed inequalities and min replaced by max at the boundary. Finally a viscosity solution is a
locally bounded function $u$ such that its u.s.c. and l.s.c. enveloppes are respectively sub- and
super-solutions of \eqref{pb:limitA}.

Since $H$ is convex we have the following representation:
\begin{lemma}\label{lem:representation}
	If $J$ satisfies \eqref{cond:J} and \eqref{eq:defH},
	the unique viscosity solution of \eqref{pb:limitA} is given by,
	\begin{equation*}
		I^A(x,t)=\min\Big\{A,tL^*\Big(\frac{\dist(x,\partial B_1)}{t}\Big)\Big\}
		\quad\text{in}\quad B_1\times(0,\infty)\,.
	\end{equation*}
\end{lemma}

\begin{proof}
	Recall that the assumptions on $J$ imply that both $H$ and $L$ are finite everywhere,
	convex, radially symmetric and super-linear, see Lemma~\ref{Lemma:properties}. We then start from the Lax-Oleinik formula
	in the bounded domain $B_1\times(0,\infty)$, see~\cite{LionsBook},
	\begin{equation*}
		I^A(x,t)=\min_{|y|\leq 1}\Big\{t L\Big(\frac{y-x}{t}\Big)+A\Big\}
		\wedge \min_{(y,s)\in \partial B_1\times(0,t)}
		\Big\{(t-s)L\Big(\frac{x-y}{t-s}\Big)\Big\}\,,
	\end{equation*}
(we denote by $a\wedge b$
the infimum of $a$ and $b$).
	Since $L$ is symmetric, we can rewrite it using $L^*$.  The fact that $L$ is nonnegative
	and $L(0)=0$ implies:
\begin{equation}
    \label{eq:minimum}
		I^A(x,t)=A
		\wedge \min_{(y,s)\in \partial B_1\times(0,t)}
		\Big\{(t-s)L^*\Big(\frac{|x-y|}{t-s}\Big)\Big\}\,.
	\end{equation}
	Since $L^*$ is increasing, the min is attained at the point $y$ such that
	$|x-y|=\dist(x,\partial B_1)$. Notice now that since $L$ is convex and $L(0)=0$,
	then for any fixed $c>0$, and variable $r>0$,
	\begin{equation*}
		(L^*)'(cr)\geq\frac{L^*(cr)}{cr}\,.
	\end{equation*}
	This implies that the function $r\mapsto L^*(cr)/r$ is increasing so that the minimum in~\eqref{eq:minimum} is attained for $s=0$ (use $c=\dist(x,\partial B_1)$ and $r=(t-s)^{-1}$ which is
	minimum	for $s=0$). Combining all these estimates, we are led to Lemma~\ref{lem:representation}.
\end{proof}


\section{Theoretical Behaviour}\label{sect:theoretical}
\setcounter{equation}{0}

The main goal of this section is to derive a theoretical bound, in terms of
the Lagrangian $L$, for the error made when approximating the solution, $u$, of~\eqref{eq:0} by solutions, $u_R$,
of the Dirichlet problem~\eqref{eq:uR}.

\begin{theorem}\label{thm:est-IR}
	If $J$ satisfies \eqref{cond:J} and \eqref{eq:defH},
	then for any fixed $x\in\R^N$ and $t>0$ there holds
	\begin{equation*}
	\lim_{R\to\infty}|u-u_R|(x,t)\leq e^{-R I_\infty(x/R,t/R)+o(1)R},
	\end{equation*}
	where the rate function is given by
	\begin{equation}\label{est:IR}
		I_\infty(x,t)=tL^*\Big(\frac{\dist(x,\partial B_1)}{t}\Big)\,.
	\end{equation}
	Moreover, for any $\theta\in(0,1)$, $0<t_1<t_2<+\infty$
	the limit is uniform in the set
	\begin{equation*}
		\big\{|x/R|\leq \theta\,,\ t/R\in(t_1,t_2)\big\}\,.
	\end{equation*}
\end{theorem}

\subsection{The transformed equation}\label{sect:transform}

Let us denote by $v_R=u-u_R$ the solution of \eqref{eq:uR} in $B_R\times(0,\infty)$ with
initial value $v_R(x,0)=0$ for $x\in B_R$ and ``boundary data'' $v_R=u$ for $|x|\geq R$.

Now we first rescale the equation both in $x$ and $t$ as follows:
\begin{equation*}
	w_R(x,t)=v_R(Rx, Rt)\quad\text{for}\quad x\in B_1,\ t\geq0\,.
\end{equation*}
Then $w_R$ satisfies a rescaled equation in the fixed ball $B_1$,  with rescaled nucleus
$J_R(x)=RJ(Rx)$:
\begin{equation*}
	\partial_t w_R(x,t)=R\big[(J_R \ast w_R)(x,t) - w_R(x,t)\big].
\end{equation*}
Let $K=K(u_0)=\Vert u_0\Vert_\infty$ and consider $\psi$ the solution of the following Dirichlet problem:
\begin{equation*}
	\begin{cases}
		\partial_t \psi_R(x)=R\big[(J_R \ast \psi_R)(x) - \psi_R(x)\big] & \text{in}\ B_1\times(0,\infty),\\
		\psi_R(x,t)=K & \text{in}\ (\R^N\setminus B_1)\times(0,\infty) ,\\
		\psi_R(x,0)=0 & \text{in}\ B_1.
	\end{cases}
\end{equation*}
Since $0\leq v_R\leq K$, then a standard comparison yields $0\leq w_R\leq\psi_R$.

In order to  estimate $\psi_R$ we follow~\cite{BarlesDaherRomano94} and perform the ``usual'' logarithmic
transform, but we have to rescale accordingly, dividing by $R$ (and not $R^2$ as it is the case
for the heat equation).	So, remembering that for $t>0$, $\psi_R>0$, let us define
\begin{equation*}
	I_R(x,t)=-\frac{1}{R}\ln (\psi_R(x,t))\,.
\end{equation*}
Then
\begin{equation*}
    \partial_t \psi_R(x,t)=-R\e^{-RI_R(x,t)}\partial_tI_R(x,t)
\end{equation*}
and
\begin{eqnarray*}
    J_R*\psi_R(x,t)-\psi_R(x,t)&=&\int_{\R^N}J_R(x-y)(\psi_R(y,t)-\psi_R(x,t))\dy\\
               &=&\int_{\R^N}RJ(Rx-Ry)\e^{-RI_R(x,t)}(\e^{R\{I_R(y,t)-I_R(x,t)\}}-1)\dy,
\end{eqnarray*}
which becomes, if we do the change of variables $y=x-s/R$,
\begin{equation*}
    \e^{-RI_R(x,t)}\int_{\R^N}J(s)(\e^{R\{I_R(x,t)-I_R(x-s/R,t)\}}-1)\d s.
\end{equation*}
We arrive at the following equation for $I_R$:
\begin{equation*}
	\partial_t I_R(x,t)=-\int_{\R^N}\Big(\e^{-R\big\{I_R(x+y/R,t)-I_R(x,t)\big\}}-1\Big) J(y)\dy\,,
\end{equation*}
which formally converges to the Hamilton-Jacobi equation,
\begin{equation}\label{eq:I}
	\partial_t I +H(\nabla I)=0\quad \text{with}\quad H(p)=\int (\e^{\,p\cdot y}-1)J(y)\dy\,.
\end{equation}
To justify convergence of $I_R$ towards the solution of \eqref{eq:I}, we have to use viscosity solutions.
This is done in the  next subsection.

\subsection{Limit Hamilton-Jacobi Equation}\label{sect:limit}

Thanks to modulus of continuity estimates proven in \cite{Chasseigne07},
and the fact that $u,u_R$ are bounded, we could extract a subsequence $\psi_{R_n}$ that converges locally
uniformly; but we shall however use the ``half-relaxed limits'' method to handle the hamiltonian.

A first problem comes from the fact that if $\psi_{R}$ approaches zero, then $I_{R}$ may not
remain bounded. Hence to avoid upper estimates for $I_R$, we use the same trick as in \cite{BarlesDaherRomano94}
which consists in modifying $I_R$ a little bit. For any $A>0$, let
\begin{equation*}
	I_R^A(x,t)=-\frac{1}{R}\ln (\psi_R(x,t) + \e^{-RA}),
\end{equation*}
which is bounded from above by $A$. Let us notice that since equation~\eqref{eq:0} is invariant under addition of
constants, $I_R^A$ satisfies the same equation as $I_R$.

\begin{proposition}
	The sequence $(I_R^A)$ converges locally uniformly in $B_1\times(0,\infty)$ as $R\to~\infty$ towards the unique
	viscosity solution $I^A$ of \eqref{pb:limitA}.
\end{proposition}

\begin{proof}
 	We introduce the half-relaxed limits,
 	\begin{equation*}
 		\overline{I}^A(x,t) := \limsup_{R\to\infty}{}^*I^A_R(x,t)=\limsup_{{(x',t')\to(x,t)}\atop{R\to\infty}}I^A_R(x',t')
 	\end{equation*}
 	and
     \begin{equation*}
         \underline{I}^A(x,t) := \liminf_{R\to\infty}{}^*I^A_R(x,t)=\liminf_{{(x',t')\to(x,t)}\atop{R\to\infty}}I^A_R(x',t'),
     \end{equation*}
 	and we shall prove that they are respectively viscosity sub- and super-solutions of the
 	limit problem \eqref{pb:limitA}. Then a uniqueness result will allow us to conclude.

	Let us take a test function $\varphi$ such that $\overline{I}^A-\varphi$ has a
	maximum at $(x_0,t_0)$.	Up to a standard modification of $\varphi$, we can assume the maximum is strict so that
	there exist sequences $R_n\to+\infty$ and $(x_n,t_n)\to(x_0,t_0)$ such that
	\begin{equation*}
	I_{R_n}^A-\varphi\text{ has a strict maximum at }(x_n,t_n)\,.
	\end{equation*}

	\textsc{Case 1: } the point $(x_0,t_0)$ is inside $B_1\times(0,\infty)$.
	Then for $n$ big enough, all the points $(x_n,t_n)$ are also inside $B_1\times(0,\infty)$
	so we may use the equation for $I_{R_n}^A$ at those points and pass to the limit.
	
	Since $\partial_t I_{R_n}^A$ is continuous, we have $\partial_t I_{R_n}^A=\partial_t\varphi$ at $(x_n,t_n)$
	and moreover, for any $z\in\R^N$,
	\begin{equation*}
		I_{R_n}^A(x_n+z,t_n)-I_{R_n}^A(x_n,t_n)\leq \varphi(x_n+z,t_n)-\varphi(x_n,t_n)\,.
	\end{equation*}
	Using this, we fix $\eps>0$ and split the equation for $I_{R_n}^A$ into two terms as follows:
	\begin{eqnarray*}
		\partial_t\varphi(x_n,t_n)&\leq& -\int_{|y|<M}\Big(\e^{-R\{\varphi(x_n+y/R_n,t_n)
		-\varphi(x_n,t_n)\}}-1\Big)	J(y)\dy \\
		&+& \Big|\int_{|y|\geq M}\big\{\psi_{R_n}(x_n+y/R_n,t_n)-\psi_{R_n}(x_n,t_n)\big\} J(y)\dy\,\Big|\,.
	\end{eqnarray*}
	Since $\psi_R$ is bounded by $K(u_0)$, we can choose $M$ big enough so that the second term is less than $\eps$,
	independently of $n$.
	
	For the first term, we write a Taylor expansion for $\varphi$ near point $x_n$:
	there exists a $\xi_n\in B_M$ (the ball of radius $M$) such that
	\begin{eqnarray*}
		\partial_t\varphi(x_n,t_n)\leq -\int_{|y|<M}\Big(\e^{-\nabla \varphi(x_n,t_n)\cdot y+\frac{1}{R_n}
		\langle D^2\varphi(\xi_n)y,y\rangle}-1\Big)J(y)\dy +\eps\,.
	\end{eqnarray*}
	Since $\xi_n$ remains in $B_M$ and  $\varphi$ is smooth we have that $D^2\varphi(\xi_n)$ remains bounded. Hence, we can pass to the limit as $n\to+\infty$:
	\begin{eqnarray*}
		\partial_t\varphi(x_0,t_0)\leq -\int_{|y|<M}\Big(\e^{-\nabla \varphi(x_0,t_0)\cdot y}-1\Big)J(y)\dy +\eps
	\end{eqnarray*}
	Since $H\big(\nabla\varphi(x_0,t_0)\big)<+\infty$, we can let $\eps\to0$ and $M\to+\infty$ to get in the limit
	\begin{eqnarray*}
		\partial_t\varphi(x_0,t_0)+\int_{\R^N}\Big(\e^{-\nabla \varphi(x_0,t_0)\cdot y}-1\Big)J(y)\dy\leq0\,.
	\end{eqnarray*}
	This shows that  $\overline{I}^A$ at $(x_0,t_0)$ is a subsolution.

	Similar calculations lead to the supersolution condition at $(x_0,t_0)$ for $\underline{I}^A$.

	\textsc{Case 2: } the point $(x_0,t_0)$ is located at the boundary, $x_0\in\partial B_1$.
	Then the sequence $(x_n,t_n)$ may either lie inside $B_1\times(0,\infty)$, or we may have
	$|x_n|\geq1$. In this last case is where the relaxed boundary condition in the viscosity sense, see Definition~\ref{def:viscosity}, comes from.

	If $x_n\in B_1$, we may use the equation as in the previous case, and we can do so even if
	$x_n\in\partial B_1$ since for $I_{R_n}^A$ the equation holds at the boundary
	(see \cite{Chasseigne07}). If on the contrary $|x_n|>1$, then $I_{R_n}(x_n,t_n)=0$
	so that in any case, one has
	\begin{equation*}
		\min\big\{\partial_t\varphi+H(\nabla\varphi)\,;\,I_{R_n}^A\big\}
		\leq0\text{ at }(x_n,t_n)\,,
	\end{equation*}
 	and we pass to the limit as $n\to+\infty$ to get the relaxed condition for $\overline{I}^A$
 	at the boundary. The converse condition for $\underline{I}^A$ is obtained by the same
 	method, with reversed inequalities.

	\textsc{Conclusion: } Using comparison between usc/lsc sub/super solutions for
	\eqref{pb:limitA}, we get the inequality $\underline{I}^A\leq\overline{I}^A$,
	which implies equality of both functions. Hence, all the sequence
	converges to the unique solution $I^A$.
\end{proof}

\subsection{Proof of Theorem \ref{thm:est-IR}}
	This result only comes from the fact that for any $A>0$, by construction
	\begin{equation*}
		\overline{I}^A=\inf(\overline{I},A),\quad \underline{I}^A=\inf(\underline{I},A),
	\end{equation*}
with
    \begin{equation*}
    \overline{I}(x,t) := \limsup_{R\to\infty}{}^*I_R(x,t),\qquad\underline{I}(x,t) := \liminf_{R\to\infty}{}^*I_R(x,t).
    \end{equation*}
The fact that $\overline{I}^A=\underline{I}^A$, together with Proposition \ref{lem:representation}
yields the result, passing to the limit as $A\to\infty$.

\section{Estimates of the Lagrangian - explicit bounds}
\label{sect:examples}
Our next aim is finding estimates for the behaviour of $L(q)$ as $|q|\to+\infty$, which in terms means estimates for $I_\infty$. Although we consider here different $J$'s the sketch of the proofs is always the same:
\begin{enumerate}
  \item At the maximum point, $p_0$, of $L(q)$ it holds $D_pH(p_0)=q$.
  \item Find estimates for $D_pH(p)$.
  \item Show that $p_0q\gg D_pH(p_0)$ and hence $L(q)\sim p_0q$.
\end{enumerate}
\subsection{Compactly supported kernels}

Let us begin with an explicit
$1$-D example:
    $$J(x)=\dfrac{1}{2}\ind_{\{|x|<1\}}(x)\,.$$
For this choice of $J$, a straightforward calculus gives $H(p)=(\sinh p)/p$, so that
\begin{equation*}
    L(q)=\sup_{p\in\R}\big\{pq-\frac{\sinh p}{p}\big\}.
\end{equation*}
At the maximum point $p_0=p_0(q)$, we have
\begin{equation}
    \label{eq:maximum}
    q=\dfrac{\cosh p_0}{p_0}-\dfrac{\sinh p_0}{p_0^2},
\end{equation}
and since $q\to\infty$ we necessarily have $p_0\to\infty$.
Hence~\eqref{eq:maximum} is equivalent to $\dfrac{e^{p_0}}{2p_0}$ as $q\to\infty$.
This implies, taking logarithms, that
\begin{equation*}
    \ln q=p_0(q)\Big(1 - \frac{\ln p_0(q)}{p_0(q)}\Big)\sim p_0(q)
\end{equation*}
so that as $q\to+\infty$, we have, using that $\sinh p_0/p_0\sim q$,
\begin{equation}
    \label{behaviour:L}
    L(q)\sim q\ln q - \frac{\sinh p_0}{p_0}+1\sim q\ln q\,.
\end{equation}
Remember that $L$ is symmetric so that the same behaviour holds as $q\to-\infty$.

This result can be easily extended to several dimensions just by using the symmetry of $H(p)$. Without loss of generality we may assume $p=\lambda e_1$, which leads to a 1-D problem.

In fact this ``$q\ln q$'' behaviour is representative of what happens in general for compactly
supported kernels in $\R^N$:

\begin{lemma}\label{lem:comp-supp}
    Let $J$ satisfy \eqref{cond:J}, \eqref{eq:defH} and $\supp(J) = \bar{B}_\eta$. Then the following behaviour holds:
    \begin{equation*}
        L^*(|q|)\sim \frac{|q|\ln |q|}{\eta}\quad\text{as}\quad |q|\to+\infty\,.
    \end{equation*}
\end{lemma}

\begin{proof}
Let $p_0=p_0(q)$ be the point where $pq-H(p)$ attains its maximum. In order to estimate $L(q)$ we have to investigate the behaviour of $q=D_pH(p)$ at $p_0$:
    \begin{equation*}
        p\cdot D_pH(p)=\int_{B_\eta} p\cdot y\e^{p\cdot y}J(y)\dy\leq \int_{B_\eta}|p|\eta\e^{|p|\eta} J(y)\dy= |p|\eta\e^{|p|\eta}.
    \end{equation*}
Taking logarithms and dividing by $\eta |p|$ we conclude that
    \begin{equation}
        \label{eq:compact:limsup}
        \limsup_{|p|\to\infty}\frac{\ln(p\cdot D_pH(p))}{\eta |p|}\leq
        \limsup_{|p|\to\infty}\left(\frac{\ln(\eta |p|)}{\eta |p|}+\frac{\eta |p|}{\eta |p|}\right)= 1.
    \end{equation}

In order to obtain an lower bound for $p\cdot D_pH(p)$ we split it into two integrals as follows:
    \begin{equation*}
        p\cdot D_pH(p)=\int_{\{p\cdot y\leq 0\}}p\cdot y\e^{p\cdot y}J(y)\dy+
                        \int_{\{p\cdot y> 0\}}p\cdot y\e^{p\cdot y}J(y)\dy.
    \end{equation*}
The first one is bounded from below by $-|p|\eta$, using that $\int_{B_\eta} J(y)\dy=1$. For the second one we define
for $\varepsilon>0$ and $\alpha<1$, the set
    \begin{equation*}
    C^+_{\varepsilon,\alpha}=\{y: \frac{\eta}{\alpha}\leq |y|\leq \eta,\    p\cdot y\geq (1-\varepsilon)|p||y|\geq 0\}.
    \end{equation*}
Since $H$ is radially symmetric and increasing with $|p|$, then
    \begin{equation*}
        D_pH(p)=\lambda(p) p\quad\text{for some}\quad \lambda(p)>0,
    \end{equation*}
and $p\cdot D_pH(p)$ has its main contribution in $C^+_{\varepsilon,\alpha}$. Hence
    \begin{equation*}
    \begin{aligned}
    \int_{\{p\cdot y> 0\}} p\cdot y\e^{p\cdot y}J(y)\dy
            &\geq |p||y|(1-\varepsilon)\int_{C^+_{\varepsilon,\alpha}}\e^{(1-\varepsilon)|p||y|}J(y)\dy\\
            &\geq |p|\frac{\eta}{\alpha}(1-\varepsilon)\e^{(1-\varepsilon)|p|\frac{\eta}{\alpha}}
            \int_{C^+_{\varepsilon,\alpha}}J(y)\dy\\
            & \geq
            C(\varepsilon,\alpha)|p|\frac{\eta}{\alpha}(1-\varepsilon)\e^{(1-\varepsilon)|p|\frac{\eta}{\alpha}}.
    \end{aligned}
    \end{equation*}
Summing up,
    \begin{equation*}
        p\cdot D_pH(p)\geq  C(\varepsilon,\alpha)|p|\frac{\eta}{\alpha}(1-\varepsilon)\e^{(1-\varepsilon)|p|\frac{\eta}{\alpha}}-|p|\eta
            \geq KC(\varepsilon,\alpha)|p|\frac{\eta}{\alpha}(1-\varepsilon)\e^{(1-\varepsilon)|p|\frac{\eta}{\alpha}},
    \end{equation*}
for some constant $K$. Therefore, arguing as in~\eqref{eq:compact:limsup} we obtain for every $\alpha$ and $\varepsilon$
    \begin{equation*}
        \liminf_{|p|\to\infty}\frac{\ln(p\cdot D_pH(p))}{\eta |p|}\geq
        \liminf_{|p|\to\infty}\left(\frac{\ln C(\varepsilon,\alpha)}{\eta |p|}+\frac{\ln(\frac{\eta}{\alpha}|p|(1-\varepsilon))}{\eta |p|}+\frac{\frac{\eta}{\alpha} |p|(1-\varepsilon)}{\eta |p|}\right)=\frac{1-\varepsilon}{\alpha}.
    \end{equation*}

Now, letting $\varepsilon\to 0$ and $\alpha\to 1$ we conclude
    \begin{equation}
        \label{eq:behaviour:DpHp:compact}
        \frac{\ln(p\cdot D_pH(p))}{\eta |p|}\sim 1,
    \end{equation}
    as $|p|\to\infty$.
The main point now is that $D_pH(p)=|D_pH(p)|\frac{p}{|p|}$ which implies that~\eqref{eq:behaviour:DpHp:compact} becomes
    \begin{equation*}
        1\sim \frac{\ln|D_pH(p)|+\ln(|p|)}{\eta |p|}\sim \frac{\ln|D_pH(p)|}{\eta |p|}.
    \end{equation*}
At the maximum point $q=D_p(H(p_0))$ we then have:
    \begin{equation*}
        p_0\cdot q=|p_0||q|\sim \frac{|q|\ln|q|}{\eta}.
    \end{equation*}

The final step consists in proving that $H(p_0)$ is negligible compared with $p_0\cdot q$. To this aim, let $J_i=\ind_{\{|y|\le\eta_i\}}$, for $i=1,2$ so that $J_1\le J\le J_2$ (remark that $\int J_i\neq 1$). Thus, from the computation done before, we obtain
    \begin{equation*}
        C_1\frac{\sinh(\eta_1|p|)}{|p|}\le H(p) \le C_2\frac{\sinh(\eta_2|p|)}{|p|}.
    \end{equation*}
These inequalities together with~\eqref{behaviour:L} yield the desired result.
\end{proof}

We are now ready to give our first concrete estimate for compactly supported kernels:

\begin{corollary}\label{cor:comp-supp}
    Let $J$ satisfy \eqref{cond:J}, \eqref{eq:defH} and $\supp(J) = \bar{B}_\eta$.
    Then the following estimate holds:
    \begin{equation*}
        \sup_{|x|\leqslant\theta R}|u-u_R|(t)\leqslant\exp\Big(-\frac{(1-\theta)R}{\eta}\ln\Big(\frac{(1-\theta)R}{t}\Big)
        + o(1)R\Big) \quad\text{as}\quad R\to\infty.
    \end{equation*}
\end{corollary}

\begin{proof}
According to the scaling $(Rx,Rt)$ described in of the Proof of Theorem \ref{thm:est-IR}, we have to estimate
$$
    I_\infty\Big(\dfrac{x}{R},\dfrac{t}{R}\Big)=\frac{t}{R}L\Big(\frac{\dist(x/R,\partial B_1)}{t/R}\Big)=\frac{t}{R}L\Big(\frac{\dist(x,\partial B_R)}{t}\Big)\,.
$$
To this aim let us keep $t>0$ fixed and take $x\in B_{\theta R}$ for some $\theta\in(0,1)$, so that $\dist(x,\partial B_R)\geq (1-\theta) R\to\infty$. We use Lemma~\ref{lem:comp-supp} to get that, as $R\to\infty$,
\begin{eqnarray*}
R I_\infty\Big(\frac{x}{R},\frac{t}{R}\Big)&\geq&R\cdot\frac{t}{R}L\Big(\frac{(1-\theta)R}{t}\Big) \\
&\sim& \frac{(1-\theta)R}{\eta}\ln\Big(\frac{(1-\theta)R}{t}\Big).\\
\end{eqnarray*}
We thus obtain the bound in $B_{\theta R}$:
$$
    \sup_{|x|\leqslant\theta R}|u-u_R|(t)\leqslant
     \exp\Big(-\frac{(1-\theta) R}{\eta}\ln\Big(\frac{(1-\theta)R}{t}\Big)+o(1)R\Big),
$$
or in a simpler form, if $t$ remains bounded:
\begin{equation*}
\sup_{|x|\leqslant\theta R}|u-u_R|(t)\leqslant \exp\Big(-\frac{(1-\theta)R}{\eta}\ln R+o(R\ln R)\Big).
\end{equation*}
We thus have a convergence rate of order $R\ln R$.
\end{proof}

This shows that we are not in the case of the heat equation for which the order
is $R^2$; here the convergence occurs at a slower speed which is due to the fact
that more paths of the process escape from the ball $B_R$. But nevertheless,
 we get a somewhat good control of the error.

\subsection{Fast exponential decay}

\begin{lemma}
	Let $J(y)=\e^{-|y|^\alpha}$, for $\alpha>1$. Then the following behaviour holds:
	\begin{equation*}
		L^*(|q|)\sim |q|(\ln |q|)^{(\alpha-1)/\alpha}\quad{as}\quad |q|\to+\infty.
	\end{equation*}
\end{lemma}

\begin{proof} We shall do the calculations in $1$-D, the adaptation to several dimensions
follows the same lines as for the case of compactly supported kernels.

We have
\begin{eqnarray*}
    H(p)&=&\int_{\mathbb{R}}(\e^{p\cdot y}-1)J(y)\dy=\int_{\mathbb{R}}\e^{p\cdot
	y-|y|^{\alpha}}\dy.
\end{eqnarray*}
Hence
\begin{eqnarray*}
	D_pH(p)&=&\int_{\mathbb{R}}y\e^{p\cdot y-|y|^{\alpha}}\dy
	\\ &=& \int_{\{p\cdot y\geq0\}}y\e^{p\cdot y-|y|^{\alpha}}\dy+
	\int_{\{p\cdot y<0\}}y\e^{p\cdot y-|y|^{\alpha}}\dy.
\end{eqnarray*}
The integral over $\{p\cdot y<0\}$ is bounded by some constant independent of $p$
since in this set, $\e^{p\cdot y}$ is less than 1, hence we will neglect it in the following estimates.

For the integral over $\{p\cdot y\geq0\}$ consider the case $p>0,\ p\to+\infty$ (the case $p\to-\infty$ being similar), and hence $y>0$.
let $y_0(p)=(p/\alpha)^{1/(\alpha-1)}$, the point where $p\cdot
y-y^{\alpha}$ attains its maximum. Then, since $p\cdot y-y^{\alpha}$ is
non-decreasing in $(y_0-1,y_0)$ and $\alpha>1$ we get
$$
    \int_{\{p\cdot y\geq0\}}y\e^{p\cdot
	y-y^{\alpha}}\dy\geq\int_{y_0-1}^{y_0}y\e^{p\cdot y-y^{\alpha}}\dy\geq
	(y_0-1)\e^{p\cdot (y_0-1)-(y_0-1)^\alpha}\geq (y_0-1)\e^{p\cdot
	y_0-y_0^\alpha}\e^{-p}.
$$
Taking logarithms,
$$
    \frac{\ln(D_pH(p))}{py_0-y_0^{\alpha}}\geq
	\frac{\ln(y_0-1)}{py_0-y_0^{\alpha}}+\frac{py_0-y_0^{\alpha}}{py_0-y_0^{\alpha}}-
	\frac{p}{py_0-y_0^{\alpha}},
$$
and hence replacing $y_0$ by $(p/\alpha)^{1/(\alpha-1)}$
$$
	\liminf_{p\to\infty}\frac{\ln(D_pH(p))}{c(\alpha)p^{\alpha/(\alpha-1)}}
	\geq 1,
$$
where
$$
	c(\alpha)=\Big(\frac{1}{\alpha}\Big)^{\frac{1}{\alpha-1}}
	\Big(1-\frac{1}{\alpha}\Big).
$$
On the other hand, let $y_\lambda=(p/\lambda\alpha)^{1/(\alpha-1)}$ the
point where $p\cdot y-\lambda y^{\alpha}$ attains its maximum. Then, for
$\lambda<1$
    \begin{eqnarray*}
    \int_{\{p\cdot y\geq0\}}y\e^{p\cdot y-y^{\alpha}}\dy &=&
    \int_{\{p\cdot y\geq0\}}y\e^{p\cdot y-\lambda	y^{\alpha}}\e^{(\lambda-1)y^{\alpha}}\dy\\
    &\leq&
     \e^{p\cdot y_\lambda-\lambda
	y_\lambda^{\alpha}}\int_{\mathbb{R}}y\e^{(\lambda-1)y^{\alpha}}\dy\\
    &=&
	C_\lambda\e^{p\cdot y_\lambda-\lambda y_\lambda^{\alpha}}.
    \end{eqnarray*}
Since $p\cdot y_\lambda-\lambda y_\lambda^{\alpha}=c_\lambda(\alpha)
p^{\alpha/(\alpha-1)}$, with
$c_\lambda(\alpha)=\lambda^{-1/(\alpha-1)}c(\alpha)$, we have taking again
logarithms and dividing by $c(\alpha)p^{\alpha/(\alpha-1)}$,
$$
    \frac{\ln(D_pH(p))}{c(\alpha)p^{\alpha/(\alpha-1)}}\leq
	\frac{c_\lambda(\alpha)}{c(\alpha)}+
    \frac{\ln(C_\lambda)}{c(\alpha)p^{\alpha/(\alpha-1)}},
$$
and hence, for any $\lambda<1$,
$$
	\limsup_{p\to\infty}\frac{\ln(D_pH(p))}{c(\alpha)p^{\alpha/(\alpha-1)}}\leq
	\frac{c_\lambda(\alpha)}{c(\alpha)}.
$$
Letting now $\lambda\to 1$, we obtain
$$
	\limsup_{p\to\infty}\frac{\ln(D_pH(p))}{c(\alpha)p^{\alpha/(\alpha-1)}}
	\leq1,
$$
and thus it turns out that
$$
    \ln(|q|)=\ln(D_pH(p))\sim c(\alpha)p^{\alpha/(\alpha-1)}.
$$
This estimate for $D_pH(p)$ yields that $L(q)$ behaves as
$|q|(\ln|q|)^{(\alpha-1)/\alpha}$ provided we show that $H(p)$ is negligible.

To achieve this last step, let us prove that $H(p)$ is at most of the order of
$D_p H(p)$. We separate the integrals in two terms as follows:
\begin{eqnarray*}
	H(p)&=&\int_{\{|y|<1\}}\e^{p\cdot y-|y|^{\alpha}}\dy + \int_{\{|y|\geq1\}}\e^{p\cdot y-|y|^{\alpha}}\dy\\
	&\leq& \e^p + \int_{\{|y|\geq1\}}y\e^{p\cdot y-|y|^{\alpha}}\dy \leq \e^p + D_pH(p)\\
	&\leq& 2D_pH(p),
\end{eqnarray*}
(of course, this is valid for $|p|$ large). Thus the lemma is proved since
\begin{equation*}
	L(q)=p_0q-H(p_0)\sim p_0q\sim |q|(\ln |q|)^{(\alpha-1)/\alpha}\,.
\end{equation*}
\end{proof}

\subsection{Critical exponential decay}\label{subsec:expo}

In this section we consider a case not covered by the results of Section \ref{sect:theoretical} since
we study the case $J(y)=\e^{-|y|}$, which is critical. Indeed, the corresponding
Hamiltonian is only finite for $|p|<1$:
\begin{equation*}
	H(p)=\frac{1}{2}\int_{\R^N}\e^{p\cdot y - |y|}\dy=
	\begin{cases}
		\frac{2}{1-p^2}& \text{for } |p|< 1,\\
		\infty& \text{for } |p|\geqslant 1.
	\end{cases}
\end{equation*}

The supremum in
\begin{equation*}
 L(q)=\sup_{p\in\R}\big\{p\cdot q-H(p)\big\}=\sup_{|p|<1}\big\{pq-\frac{2}{1-p^2}\Big\}
\end{equation*}
is obtained for $q=\dfrac{4p}{(1-p^2)^2}$.
Here, $q\to\infty$ corresponds to $p\to\pm1$ which means that
\begin{equation*}
 |1-p^2|\sim2|q|^{1/2},
\end{equation*}
so that
\begin{equation*}
L(q)\sim q\pm\frac{1}{2q^{1/2}}\sim q\,.
\end{equation*}

However, Theorem \ref{thm:est-IR} does not apply as such here, since $H$ takes infinite values.
So we only give a formal estimate for the sake of comparison:
\begin{equation*}
 |u-u_R|(x,t)\leqslant \e^{-(1-\theta)R+o(1)R}.
\end{equation*}
We shall adress the details of this case in a forthcoming paper, which imply non trivial
adaptations of Section \ref{sect:theoretical}.


\section{Generalization to infinite activity jump diffusions}\label{sect:levy}

In this section, we briefly explain how to extend our results to a class of singular kernels.
We consider here functions $J$ satisfying:
\begin{equation*}
	\int_{\R^N}\Big(1\wedge |y|^2\Big)J(y)\dy<+\infty\,,
\end{equation*}
that is, we are interested in L\'evy measures with density $J$. Typical examples of such L\'evy measure are:
\begin{equation*}
	J(y)=\frac{1}{|y|^{N+\alpha}}\,,\ J(y)=\frac{\e^{-|y|}}{|y|^{N+\alpha}}\,,\
	J(y)=\frac{1}{|y|^{N+\alpha}}\ind_{\{|y|<1\}}(y)\,,
\end{equation*}
where $0<\alpha<2$. The first example is related to the well-known fractional Laplacian,
while the second example is called ``tempered $\alpha$-stable law'' among the probability community.
The third example is singular at the origin, but compactly supported.

However, the equation has to be understood in a special way:
since $J$ is not integrable near the origin, the equation should contain
an extra term (called ``corrector'') in order to give sense to the integral term:
\begin{equation*}
	\partial_tu = \int_{\R^N}\Big\{u(x+y)-u(x)-\ind_{\{|y|<1\}}(\nabla u\cdot y)\Big\}J(y)\,\dy\,.
\end{equation*}
Of course, if $u$ were $\mathrm{C}^2$-smooth, the integrand would be close to $D^2u$ for
$y$ small, and everything would be integrable. But since we do not know \textit{a priori} the
regularity of $u$, we have to replace it by some smooth test-function and this is where viscosity solutions
enter into play. We refer to \cite{BarlesChasseigneImbert07,BarlesImbert07} for precise
definitions and properties of visosity solutions in presence of L\'evy-type non-local terms.

With this tool, everthing works exactly as we did for the non-singular case, except that
the new Hamiltonian also involves a corrector term:
\begin{equation}\label{def:ham-sing}
	H(p)=\int_{\R^N}\big\{\e^{\,p\cdot y}-1 -(p\cdot y)\ind_{\{|y|<1\}}\big\}J(y)\dy\,,\quad p\in\R^N\,.
\end{equation}
Notice that integrability near the origin comes from the assumption on $J$ (since
$\e^py-1-py\sim (py)^2/2$ for $y\sim 0$), but we have to face the same integrability condition
at infinity: we shall assume that $H(p)<+\infty$ everywhere in $\R^N$.

So, the case of fractional Laplace operators is not covered here
since in this case, $H\equiv+\infty$; but the third example above can be dealt with. In view
of Section \ref{subsec:expo}, we hope also to cover the $\alpha$-stable law case soon.

For these adaptations, using the fact that the various manipulations that are used in
Section \ref{sect:theoretical} are valid for viscosity solutions, it can be
checked that the same limit Hamilton-Jacobi equation is obtained:
\begin{equation*}
 \partial_t I_\infty+H(\nabla I_\infty)=0\,.
\end{equation*}

Now, when one takes a look at the singular Hamiltonian $H$, it appears clearly that the behaviour
at infinity (i.e., for $|p|$ large) does not depend on the corrector term which is of lower order.
Another way of understanding this is to think in terms of the total amount of process that can
escape the ball $B_R$: small jumps are not responsible for that behaviour (in first approximation).
The main important contribution comes from big jumps, related to the tail of $J$.

To illustrate this heuristic remark, let us consider the case of a compactly supported
measure with a singularity:
we consider the $1$-D kernel
\begin{equation*}
	J(y)=\frac{1}{|y|^{1+\alpha}}\ind_{[-1,1]}(y)\,,\quad\text{where}\quad \alpha\in(0,2)\,.
\end{equation*}
We do the same computations with
\begin{equation*}
	D_pH(p)=\int_{-1}^{1}\frac{\e^{p\cdot y}}{|y|^\alpha}\dy\,,
\end{equation*}
which gives on the one hand	$D_p H(p)\leq c(\alpha)\e^{p}$, so that
\begin{equation*}
	\limsup_{p\to+\infty}\frac{q}{c(\alpha)\e^p}\leq1\,,
\end{equation*}
and thus the maximal behaviour of $L(q)$ is:
\begin{equation*}
	\limsup_{p\to\infty}\frac{L(q)}{|q|\ln|q|}\leq c(\alpha)\,.
\end{equation*}

On the other hand we use the fact that the Legendre-Fenchel is non-increasing, which is
obvious from the sup formula relating $L$ and $H$: since
\begin{equation*}
	H(p)=\int_{-1}^{1}\frac{\e^{p\cdot y}}{|y|^{1+\alpha}}\dy\geq \widehat{H}(p)=
	\int_{-1}^{1}\frac{\e^{p\cdot y}}{1+|y|^{1+\alpha}}\dy\,,
\end{equation*}
we know that the corresponding $L$ satisfies: $L(q)\geq \widehat{L}(q)$. But the modified kernel is a smooth
and compactly supported function for which the behaviour is known:
\begin{equation*}
	\widehat{L}(q)\sim c'(\alpha)|q|\ln |q|\,.
\end{equation*}
Finally we get that for some constants $c_1,c_2>0$,
\begin{equation*}
	c_1|q|\ln|q|\leq L(q)\leq c_2|q|\ln|q|\,.
\end{equation*}

The conclusion is that the presence of singularities at the origin
does not change the scale essentially. Similar examples can be derived for
non compactly supported kernels.
%
%
%
%

\section{Numerical experiments}\label{sect:numerical}

We gather now some numerical examples, with the aim of illustrating the convergence theorem~\ref{thm:est-IR}.

We try to approximate solutions of~\eqref{eq:0}, with initial datum $u_0=1$. The advantage of taking this datum is that the solution of the problem is $u\equiv 1$. We have an exact solution that we do not compute numerically.

In order to approximate $u_R$ we use a fixed mesh scheme that discretizes the interval $[-R,R]$, and an ODE integrator provide by {\sc Matlab}$^{\tiny \circledR}$ to integrate in time up to the fixed time $t_0=0.1$. For each time-step, the integral involved in the term $J*u_R$ is approximated by the classical trapezoidal rule.

%

\subsection{Compactly supported kernels}
We consider here a compactly supported nucleus
$$J(y)=-\frac{3}{32}(x-2)(x+2).$$
\begin{figure}[ht!]
  \begin{center}
  \includegraphics[width=10cm]{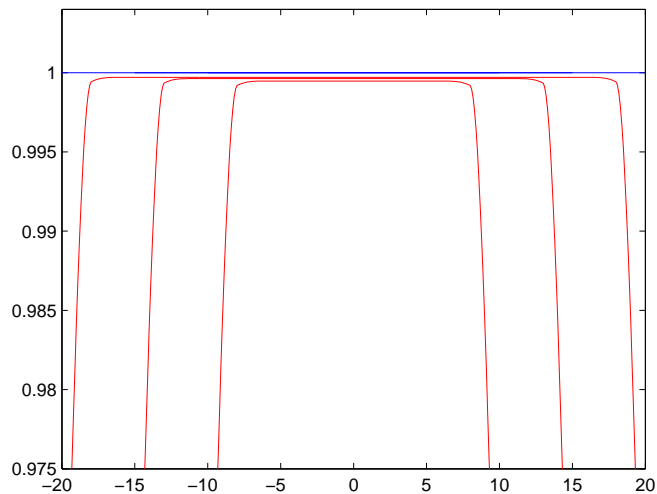}
  \caption{$J$ compactly supported}
  \label{fig:compact}
\end{center}
\end{figure}
 In Figure~\ref{fig:compact} we show convergence of $u_R$ to $u=1$. We observe that $u_R$  increases as $R$ increases (the values of $R$ are $R=10, 15, 20$).


\subsection{Gaussian example}

Consider the case $$J(y)=\frac{1}{\sqrt{2\pi}}\e^{-y^2/2}.$$
 Using the same ideas as in Section~\ref{sect:examples} we compute
\begin{eqnarray*}
 H(p)&=&\frac{1}{\sqrt{2\pi}}\int_\R \e^{\,py-y^2/2}\dy-1\\
&=&\frac{1}{\sqrt{2\pi}}\int_\R \e^{-(p-y)^2/2}\cdot \e^{\,p^2/2}\dy-1\\
&=&\e^{\,p^2/2}-1.
\end{eqnarray*}
Now, $D_pH(p)=p\e^{\,p^2/2}$ and $L(q)$ is obtained for $q=p\e^{\,p^2/2}$. For $q\to\infty$, we have $\ln q\sim\ln p+ p^2/2\sim p^2/2$, thus $p\sim 2(\ln q)^{1/2}$ so that
\begin{equation*}
L(q)\sim 2q(\ln q)^{1/2}\,.
\end{equation*}

In the picture on the left of Figure~\ref{fig:gauss} we show again convergence for different values of $R$ ($R=10,15,20$). The picture on the right represents for $x$ and $t$ fixed the convergence of this point to the exact solution.
\begin{center}
  \begin{figure}[ht!]
  \subfigure{%
  \includegraphics[width=8cm]{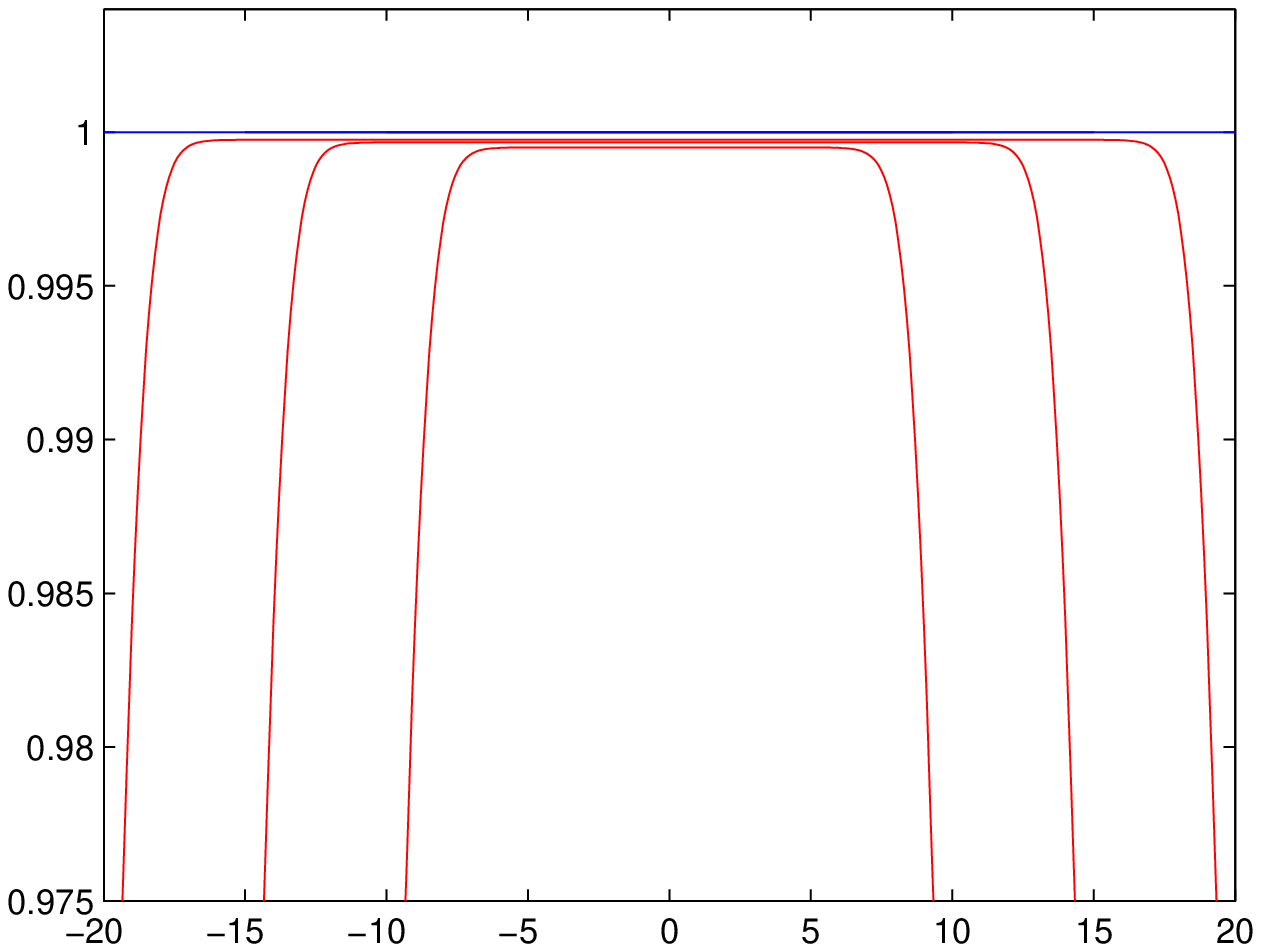}}
\subfigure{%
  \includegraphics[width=8cm]{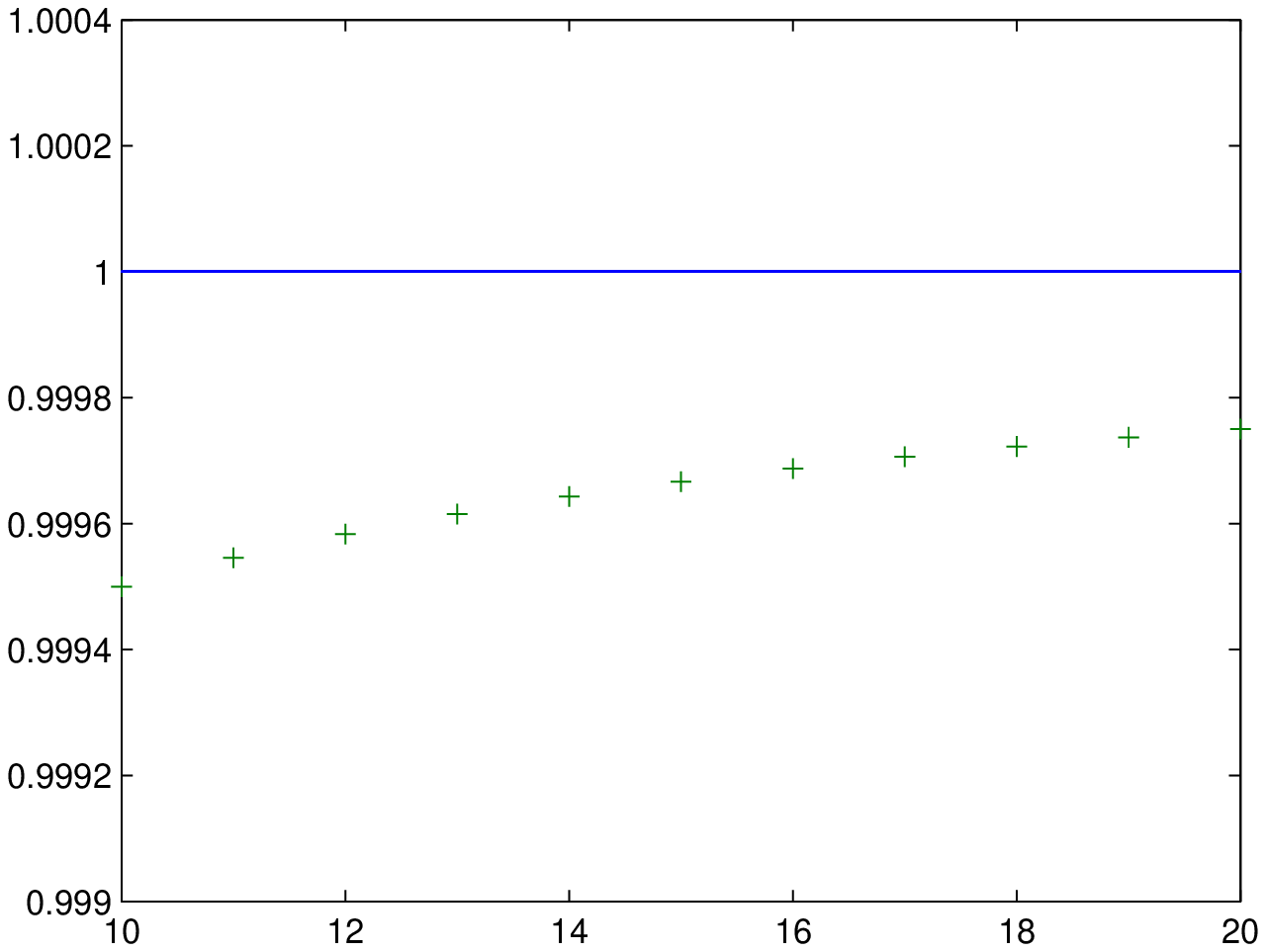}}
  \caption{$J$ gaussian, $R=10, 15, 20$}
\label{fig:gauss}
\end{figure}
\end{center}

\subsection{Critical exponential decay}

Finally we consider the case of a nucleus with critical exponential decay,
$$
J=\frac{1}{2}\e^{-|y|}.
$$
As we mentioned before, this case is not covered by the results of Section~\ref{sect:theoretical}, but it shall illustrate how convergence works for a general case:

\begin{center}
  \begin{figure}[ht!]
  \subfigure{%
  \includegraphics[width=8cm]{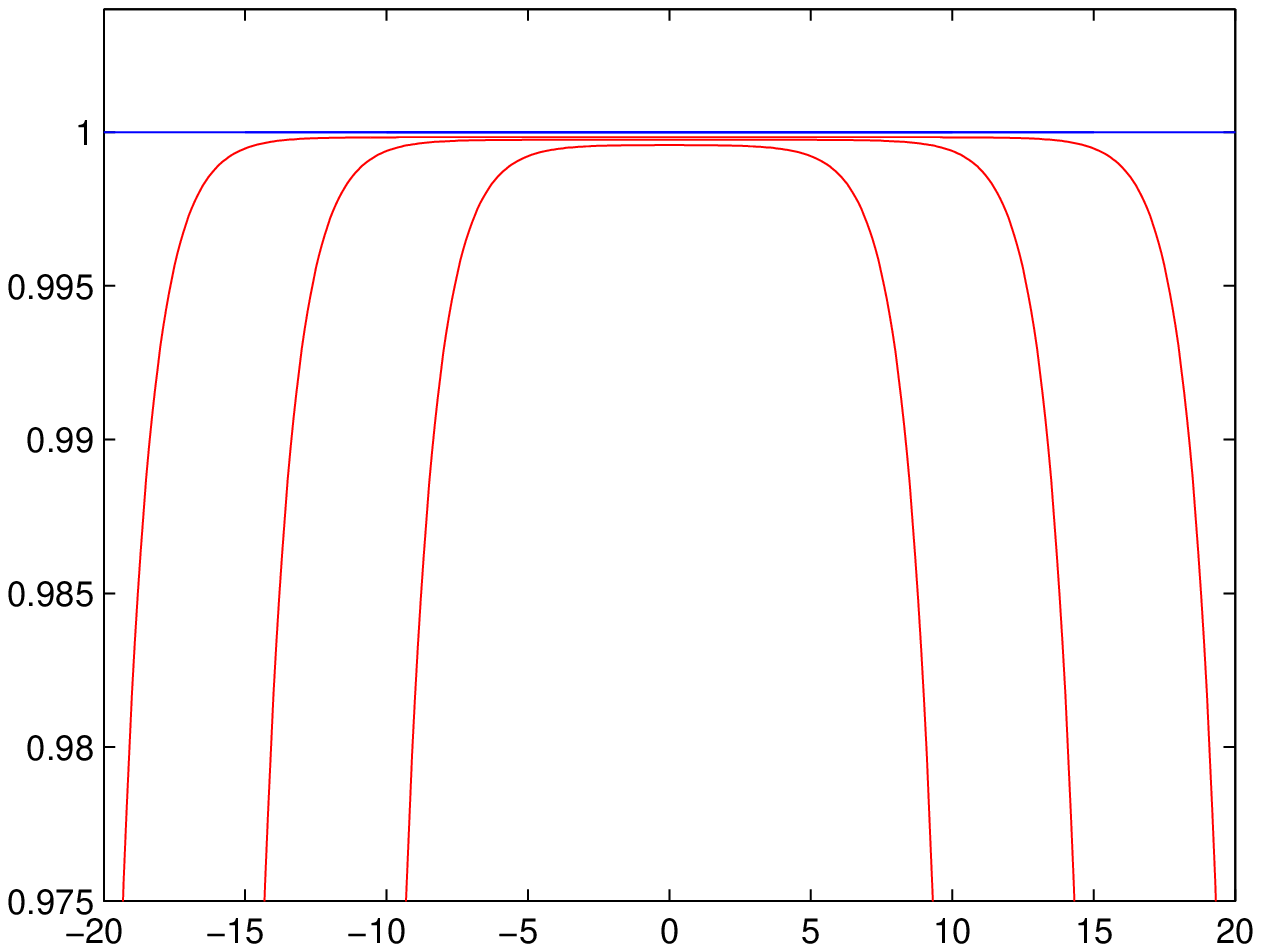}}
\subfigure{%
  \includegraphics[width=8cm]{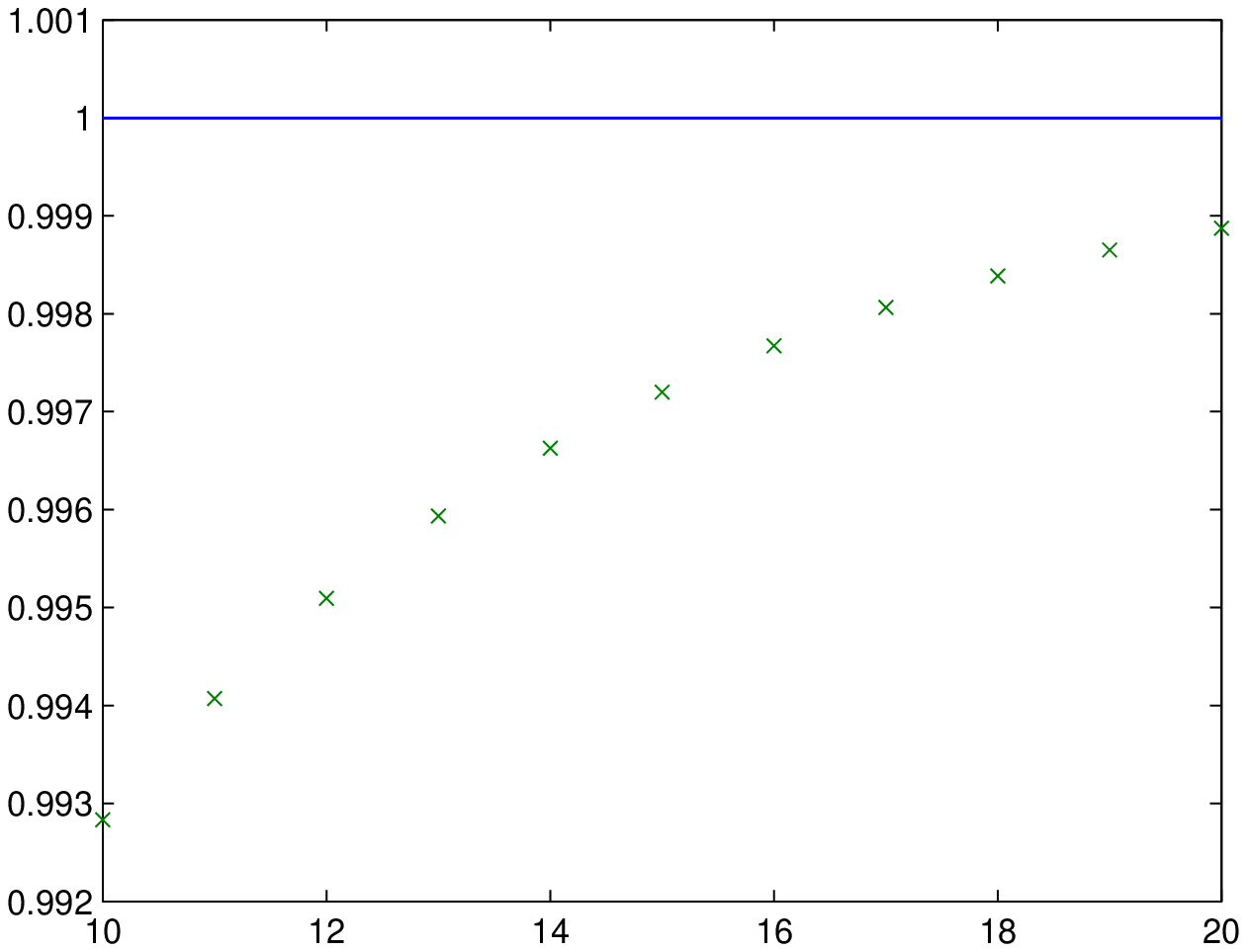}}
  \caption{$J$ critical exponential decay, $R=10, 15, 20$}
\end{figure}
\end{center}

The picture on the right shows convergence of a point that moves with $R$;~\i.e. we take a point of the form $(\theta R,t)$ with $t=0.1$ fixed and $\theta=0.8$.


\end{document}